\documentclass[reqno]{amsart}
\usepackage{amsfonts,amsmath,amsthm,amssymb,mathrsfs}
\usepackage{color}
\usepackage{amsmath,amssymb,amsfonts,bm}
\usepackage{graphicx}
\usepackage{newunicodechar}
\usepackage{xcolor}
\numberwithin{equation}{section}
\usepackage[pdfstartview=FitH,colorlinks=true]{hyperref}
\hypersetup{urlcolor=red, linkcolor=blue,citecolor=red}
\allowdisplaybreaks

\theoremstyle{plain}

\newtheorem{thm}{Theorem}[section]

\newtheorem{prop}[thm]{Proposition}

\newtheorem{lemma}[thm]{Lemma}
\newtheorem{remark}[thm]{Remark}

\begin{document}
\title[Landau equation]
{A remark about time-analyticity of the linear Landau equation with soft potential
}

\author[C.-J. Xu \& Y. Xu]
{Chao-Jiang Xu and Yan Xu}

\address{Chao-Jiang Xu and Yan Xu
\newline\indent
School of Mathematics and Key Laboratory of Mathematical MIIT,
\newline\indent
Nanjing University of Aeronautics and Astronautics, Nanjing 210016, China
}
\email{xuchaojiang@nuaa.edu.cn; xuyan1@nuaa.edu.cn}

\date{\today}



\begin{abstract}
In this note, we study the Cauchy problem of the linear spatially homogeneous Landau equation with soft potentials.
We prove that the solution to the Cauchy problem enjoys the analytic regularizing effect of the time variable with an $L^2$ initial datum for positive time. So that the smoothing effect of Cauchy problem for the linear spatially homogeneous Landau equation with soft potentials is similar to the heat equation.
\end{abstract}

\maketitle

\section{Introduction}
\par The Cauchy problem of spatially homogenous Landau equation reads 
\begin{equation}\label{1-1}
\left\{
\begin{aligned}
 &\partial_t F=Q(F, F),\\
 &F|_{t=0}=F_0,
\end{aligned}
\right.
\end{equation}
where $F=F(t, v)\ge0$ is the density distribution function at time $t\ge0$, with the velocity variable $v\in\mathbb R^3$. The Landau bilinear collision operator is defined by
\begin{equation}\label{Q}
    Q(G, F)(v)=\sum_{j, k=1}^3\partial_j\bigg(\int_{\mathbb R^3}a_{jk}(v-v_*)[G(v_*)\partial_kF(v)-\partial_kG(v_*)F(v)]dv_*\bigg),
\end{equation}
with
\begin{displaymath}
   a_{jk}(v)=(\delta_{jk}|v|^2-v_jv_k)|v|^\gamma,\quad \gamma\ge-3,
\end{displaymath}
is a symmetric non-negative matrix such that
$$\sum_{j, k=1}^3a_{jk}(v)v_jv_k=0.$$
Here, $\gamma$ is a parameter which leads to the classification of the hard potential if $\gamma>0$, Maxwellian molecules if $\gamma=0$, soft potential if $-3<\gamma<0$ and Coulombian potential if $\gamma=-3$.

The Landau equation was introduced as a limit of the Boltzmann equation when the collisions become grazing in~\cite{D-1, V-1}. In the hard potential case, the existence, and the uniqueness of the solution to the Cauchy problem for the spatially homogeneous Landau equation have been addressed by Desvillettes and Villani in~\cite{D-2, V-2}. Meanwhile, they also proved the smoothness of the solution is $C^\infty(]0, \infty[; \mathcal S(\mathbb R^3))$. The analytic and the Gevrey regularity of
the solution for any $t>0$ have already been studied in~\cite{C-1, C-2}.

We shall study the linearization of the Landau equation \eqref{1-1} near the Maxwellian distribution
\begin{equation*}
    \mu(v)=(2\pi)^{\frac32}e^{-\frac{|v|^2}{2}}.
\end{equation*}
Considering the fluctuation of the density distribution function
\begin{equation*}
    F(t, v)=\mu(v)+\sqrt\mu(v)f(t, v),
\end{equation*}
since $Q(\mu, \mu)=0$, the Cauchy problem \eqref{1-1} takes the form
\begin{equation*}
\left\{
\begin{aligned}
 &\partial_tf+\mathcal Lf=\Gamma(f, f),\\
 &f|_{t=0}=f_0,
\end{aligned}
\right.
\end{equation*}
with $F_0=\mu+\sqrt\mu f_0$, where
\begin{equation*}
    \Gamma(f, f)=\mu^{-\frac12}Q(\sqrt\mu f, \sqrt\mu f),
\end{equation*}
\begin{equation*}
\mathcal{L}=\mathcal{L}_1+\mathcal{L}_2,\ \ \ \mbox{with}\ \ \    \mathcal L_1f=-\Gamma(\sqrt\mu, f), \quad \mathcal L_2f=-\Gamma(f, \sqrt\mu).
\end{equation*}

The spatially homogeneous Landau equation and non-cutoff Boltzmann equation in a close-to-equilibrium framework have been studied in~\cite{L-1} and the Gelfand-Shilov smoothing effect has been proved in~\cite{L-2, M-1}. Guo~\cite{G-1} constructed the classical solution for the spatially inhomogeneous Landau equation near a global Maxwellian in a periodic box. The smoothness of the solutions has been studied in~\cite{C-3, H-1, L-3}. In addition, the analytic smoothing effect of the velocity variable for the nonlinear Landau equation has been treated in~\cite{L-4, M-2}. The variant regularity results under a close to equilibrium setting have been considered in~\cite{C-4, C-5, S-1}.

In this work, we consider the Cauchy problem of the linear Landau equation, such as
\begin{equation}\label{1-2}
\left\{
\begin{aligned}
 &\partial_tf+\mathcal Lf=g,\\
 &f|_{t=0}=f_0,
\end{aligned}
\right.
\end{equation}
where $g$ is a analytic function with respect to the variable $t$ and $v$. The diffusion part $\mathcal L_1$ is written as follows
\begin{equation}\label{1-3}
    \mathcal L_1f=-\nabla_v\cdot[A(v)\nabla_vf]+\left(A(v)\frac{v}{2}\cdot\frac{v}{2}\right)f-\nabla_v\cdot\left[A(v)\frac{v}{2}\right]f,
\end{equation}
with $A(v)=(\bar a_{jk})_{1\le j, k\le3}$ is a symmetric matrix, and
\begin{equation*}
    \bar a_{jk}=a_{jk}*\mu=\int_{\mathbb R^3}\left(\delta_{jk}|v-v'|^2-(v_j-v'_j)(v_k-v'_k)\right)|v-v'|^\gamma\mu(v')dv'.
\end{equation*}

We say that $u\in\mathcal A(\Omega)$ is an analytic function, where $\Omega\subset\mathbb R^n$ is an open domain, if $u\in C^\infty(\Omega)$ and there exists a constant $C$ such that for all multi-indices $\alpha\in\mathbb N^n$,
\begin{equation*}
    \|\partial^\alpha u\|_{L^\infty(\Omega)}\le C^{|\alpha|+1}\alpha!.
\end{equation*}
Remark that, by using the Sobolev embedding, we can replace the $L^\infty$ norm by the $L^2$ norm, or norm in any Sobolev space in the above definition.

We study the linear Landau equation \eqref{1-2}, with $-3<\gamma<0$, and show that the solution to the Cauchy problem \eqref{1-2} with the $L^2(\mathbb R^3)$ initial datum enjoys the analytic regularizing effect of the time variable. The main result reads as follows.

\begin{thm}\label{thm}
    For the soft potential $-3<\gamma<0$, for any $T>0$ and the initial datum $f_0\in L^2(\mathbb R^3)$. Let $f$ be the solution of the Cauchy problem \eqref{1-2}, then there exists a constant $C>0$ such that for any $k\in\mathbb N$, we have
    \begin{equation}\label{1-4}
        \|\partial^k_tf(t)\|_{L^2(\mathbb R^3)}\le\frac{C^{k+1}}{t^k}k!,\quad\forall t\in]0,T].
    \end{equation}
\end{thm}

For the linear operator with only the diffusion part of $\mathcal{L}_1$, the paper~\cite{L-5} prove that the Cauchy problem \eqref{1-2} admits a unique weak solution, and the solution satisfies for any $\alpha\in\mathbb N^3$, $\tilde t=\min(t,1)$,
    \begin{equation*}
        \|\tilde t^{\frac{|\alpha|}{2}}\langle\cdot\rangle^{\frac{\gamma|\alpha|}{2}}\partial^\alpha f(t)\|_{L^2(\mathbb R^3)}\le C^{|\alpha|+1}\alpha!,\quad\forall t>0.
    \end{equation*}
With the similar computation, one can obtain the same analytical results as above, then using again the equation of \eqref{1-2}, on have
$$
f\in C^\infty(]0, +\infty[; \mathcal A(\mathbb{R}^3)).
$$
So that we just need to prove \eqref{1-4} for the smooth solution of Cauchy problem \eqref{1-2}.

\section{Analysis of the Landau Linear Operator}

In the following, the notation $A\lesssim B$ means there exists a constant $C>0$ such that $A\le CB$. For simplicity, with $\gamma\in\mathbb R$, we denote the weighted Lebesgue spaces
\begin{equation*}
    \|\langle\cdot\rangle^\gamma f\|_{L^p(\mathbb R^3)}=\|f\|_{p, \gamma},\quad 1\le p\le\infty,
\end{equation*}
where we use the notation $\langle v\rangle=(1+|v|^2)^{\frac12}$. And for the
matrix $A$ defined in \eqref{1-3}, we denote
\begin{equation*}
    \|f\|^2_A=\sum_{j, k=1}^3\int\left(\bar a_{jk}\partial_jf\partial_kf+\frac14\bar a_{jk}v_jv_kf^2\right)dv.
\end{equation*}

From corollary 1 in~\cite{G-1}, for $\gamma>-3$, there exists a constant $C_1>0$ such that
\begin{equation}\label{2-1}
    \|f\|^2_{A}\ge C_1\left(\|\mathbf P_v\nabla f\|^2_{2, \frac{\gamma}{2}}+\|(\mathbf I-\mathbf P_v)\nabla f\|^2_{2, 1+\frac{\gamma}{2}}+\|f\|^2_{2, 1+\frac{\gamma}{2}}\right),
\end{equation}
where for any vector-valued function $G(v)=(G_1, G_2, G_3)$ define the projection to the vector $v=(v_1, v_2, v_3)\in\mathbb R^3$ as
\begin{equation*}
    (\mathbf P_vG)_j=\sum_{k=1}^3G_kv_k\frac{v_j}{|v|^2},\quad 1\le j\le3.
\end{equation*}
Since
\begin{equation*}
    \nabla f=\mathbf P_v\nabla f+(\mathbf I-\mathbf P_v)\nabla f,
\end{equation*}
combining the inequality \eqref{2-1}, we have
\begin{equation}\label{2-2}
    \|f\|_{A}\ge C_1\left(\|\nabla f\|_{2,\frac{\gamma}{2}}+\|f\|_{2, 1+\frac{\gamma}{2}}\right).
\end{equation}

For later use, We need the following results for the coefficients to the linear Landau operator, which have been proved in~\cite{L-5}.

\begin{lemma}[~\cite{L-5}]
    For any $\beta\in\mathbb R^3$ with $|\beta|\ge1$ and $\bar a_{jk}$ was defined in \eqref{1-3} with $-3<\gamma<0$, then we have
    \begin{equation}\label{a}
        |\partial^\beta\bar a_{jk}(v)|\lesssim\langle v\rangle^{\gamma+1}\sqrt{\beta!}.
    \end{equation}
    Moreover, for any $\beta\in\mathbb R^3$,
    \begin{equation}\label{A}
    \begin{split}
        &\left|\partial^\beta\bigg(\sum_{j, k=1}^3\partial_ja_{jk}*(v_k\mu)\bigg)\right|\lesssim\langle v\rangle^{\gamma+1}(|\beta|+1)\sqrt{\beta!},\\
        &\left|\partial^\beta\bigg(\sum_{j, k=1}^3\bar a_{jk}v_jv_k\bigg)\right|\lesssim\langle v\rangle^{\gamma+1}(|\beta|+1)\sqrt{\beta!}.
    \end{split}
    \end{equation}
\end{lemma}
And
\begin{lemma}[~\cite{L-5}]
    Let $f_1, f_2\in\mathcal S(\mathbb R^3)$, $\bar a_{jk}$ was defined in \eqref{1-3} with $-3<\gamma<0$. For any $\beta\in\mathbb R^3$, we have
    \begin{equation}\label{0}
        \left|\sum_{j, k=1}^3(\partial^\beta\bar a_{jk}\partial_kf_1, \partial_jf_2)_{L^2(\mathbb R^3)}\right|\lesssim\sqrt{\beta!}\|f_1\|_A\|f_2\|_A.
    \end{equation}
\end{lemma}

By using the results of the coefficients to the linear Landau operator in~\cite{L-5}, we can obtain the following estimates. Firstly, for any $\gamma>-3$ and $\delta>0$, we have
\begin{equation}\label{2-4}
    \int_{\mathbb R^3}|v-w|^{\gamma}e^{-\delta|w|^2}dw\lesssim\langle v\rangle^{\gamma}.
\end{equation}

\begin{lemma}\label{lemma2.1}
    Let $f\in\mathcal S(\mathbb R^3)$, and $-3<\gamma<0$, then for any $0<\epsilon_1<1$, there exists a constant $C_{\epsilon_1}>0$ such that
    \begin{equation*}
       (1-\epsilon_1)\|f\|^2_{A}\le (\mathcal L_1f,  f)_{L^2}+C_{\epsilon_1}\|f\|^2_{2, \frac{\gamma}{2}}.
    \end{equation*}
\end{lemma}

\begin{proof}
    By using the representation \eqref{1-3}, and integrating by parts, we have
    \begin{equation*}
        \begin{split}
            -(\mathcal L_1f, f)_{L^2}&=-\int_{\mathbb R^3}\left(\bar a_{jk}\partial_jf\partial_kf+\frac14\bar a_{jk}v_jv_kf^2\right)-\frac12\int_{\mathbb R^3}\partial_j(\bar a_{jk}v_k)f^2\\
            &=-\|f\|^2_{A}+R_0.
        \end{split}
    \end{equation*}
    Since
    \begin{equation}\label{2-3}
        \sum_{j}a_{jk}v_j=\sum_{k}a_{jk}v_k=0,
    \end{equation}
    we have
    \begin{equation*}
        \partial_j(\bar a_{jk}v_k)=\partial_j(a_{jk}*(v_k\mu))=\partial_ja_{jk}*(v_k\mu).
    \end{equation*}
    Therefore from \eqref{A} and the Cauchy-Schwarz inequality, it follows that
    \begin{equation*}
        |R_0|\lesssim\int_{\mathbb R^3}\langle v\rangle^{\gamma+1}f^2(v)dv\lesssim\|f\|_{2, \frac{\gamma}{2}}\|f\|_{2, 1+\frac{\gamma}{2}},
    \end{equation*}
    then by using \eqref{2-2} and the Cauchy-Schwarz inequality, for any $0<\epsilon_1<1$, we have
    \begin{equation*}
        |R_0|\le C_2\|f\|_{2, \frac{\gamma}{2}}\|f\|_{A}\le\epsilon_1\|f\|^2_{A}+\frac{4C^2_2}{\epsilon_1}\|f\|^2_{2, \frac{\gamma}{2}}.
    \end{equation*}
    Let $C_{\epsilon_1}=\frac{4C^2_2}{\epsilon_1}$, then we can conclude
    \begin{equation*}
        (1-\epsilon_1)\|f\|^2_{A}\le (\mathcal L_1f,  f)_{L^2}+C_{\epsilon_1}\|f\|^2_{2, \frac{\gamma}{2}}.
    \end{equation*}
\end{proof}

\begin{prop}\label{prop L}
    Let $f_1, f_2\in\mathcal S(\mathbb R^3)$ and $-3<\gamma<0$, then there exists a constant $C_2>0$ such that
    \begin{equation*}
        |(\mathcal L_1f_1, f_2)_{L^2}|\le C_2\|f_1\|_A\|f_2\|_A.
    \end{equation*}
\end{prop}
\begin{proof}
    By using the representation \eqref{1-3}, and integrating by parts, we have
    \begin{equation*}
        \begin{split}
            (\mathcal L_1f_1, f_2)_{L^2}&=\int_{\mathbb R^3}\bar a_{jk}\partial_jf_1\partial_kf_2+\frac14\int_{\mathbb R^3}\bar a_{jk}v_jv_kf_1f_2\\
            &\quad+\frac12\int_{\mathbb R^3}\partial_j(\bar a_{jk}v_k)f_1f_2\\
            &=R_1+R_2+R_3.
        \end{split}
    \end{equation*}
    Since $-3<\gamma<0$, by the inequality \eqref{0}, we obtain
    \begin{equation*}
        \begin{split}
            |R_1|\lesssim\|f_1\|_A\|f_2\|_A.
        \end{split}
    \end{equation*}
    For the term $R_2$ and $R_3$, using \eqref{2-3}, then from \eqref{a} and \eqref{A}, it follows that
    \begin{equation*}
        |R_2|+|R_3|\lesssim\int_{\mathbb R^3}\langle v\rangle^{\gamma+1}|f_1(v)f_2(v)|dv,
    \end{equation*}
    then using the Cauchy-Schwarz inequality and \eqref{2-2}, we have
    \begin{equation*}
        \begin{split}
            |R_2|+|R_3|\lesssim\|f_1\|_{2, 1+\frac{\gamma}{2}}\|f_2\|_{2, 1+\frac{\gamma}{2}}\lesssim\|f_1\|_A\|f_2\|_A.
        \end{split}
    \end{equation*}
    Finally, combining $R_1-R_3$ to get
    \begin{equation*}
        |(\mathcal L_1f_1, f_2)_{L^2}|\le C_2\|f_1\|_A\|f_2\|_A.
    \end{equation*}
\end{proof}

Now, we shall estimate $(\mathcal L_2f_1, f_2)_{L^2}$. Firstly, we give the representation of the operator $\mathcal L_2$. For $f\in\mathcal S(\mathbb R^3)$, from \eqref{Q}, it follows that
    \begin{equation}\label{K}
        \begin{split}
            \mathcal L_2f&=-\mu^{-\frac12}Q(\sqrt\mu f,\mu)\\
            &=\mu^{-\frac12}\partial_j\left(\int_{\mathbb R^3}a_{jk}(v-v')\left[\mu^{\frac12}(v')f(v')v_k+\partial_k\left(\mu^{\frac12}f\right)(v')\right]dv'\mu(v)\right)\\
            &=\mu^{-\frac12}\partial_j\left[\mu\left(a_{jk}*(v_k\mu^{\frac12}f)+\partial_ka_{jk}*(\mu^{\frac12}f)\right)\right].
        \end{split}
    \end{equation}

\begin{prop}\label{prop K}
    Let $f_1, f_2\in\mathcal S(\mathbb R^3)$ and $-3<\gamma<0$, then there exists a constant $C_{3}>0$ such that
    \begin{equation}\label{L2}
        |(\mathcal L_2f_1, f_2)_{L^2}|\le C_3\left(\|f_1\|_{2, \frac{\gamma}{2}}\|f_2\|_A+\|f_1\|_A\|f_2\|_{2, \frac{\gamma}{2}}\right).
    \end{equation}
\end{prop}
\begin{proof}
    Using integration by parts with \eqref{K}, we have
    \begin{equation*}
        \begin{split}
            (\mathcal L_2f_1, f_2)_{L^2}&=-\left(a_{jk}*(v_k\mu^{\frac12}f_1), \mu^{\frac12}\left(\frac{v_j}{2}f_2+\partial_jf_2\right)\right)_{L^2}\\
            &\quad+\left(\partial_{jk}a_{jk}*(\mu^{\frac12}f_1), \mu^{\frac12}f_2\right)_{L^2}-\left(\partial_ka_{jk}*(\mu^{\frac12}f_1), \mu^{\frac12}v_jf_2\right)_{L^2}\\
            &=I_1+I_2+I_3.
        \end{split}
    \end{equation*}
    Since
    \begin{equation}\label{2-5}
        |\partial^\alpha a_{jk}(v)|\le|v|^{\gamma+2-|\alpha|}, \quad \forall\alpha\in\mathbb N^3,
    \end{equation}
    and 
    \begin{equation}\label{mu}
        \langle v\rangle^\beta\mu^\rho(v)\in L^\infty(\mathbb R^3),\quad \forall\beta\in\mathbb R, \rho>0,
    \end{equation}
    by Cauchy-Schwarz inequality, we obtain
    \begin{equation*}
        \begin{split}
            \left|a_{jk}*(v_k\mu^{\frac12}f_1)\right|&\le\int_{\mathbb R^3}|v-v'|^{\gamma+2}\langle v'\rangle^{1-\frac{\gamma}{2}}\mu^{\frac12}(v')\langle v'\rangle^{\frac{\gamma}{2}}|f_1(v')|dv'\\
            &\lesssim\left(\int_{\mathbb R^3}|v-v'|^{2(\gamma+2)}\mu^{\frac12}(v')dv'\right)^{\frac12}\|f_1\|_{2, \frac{\gamma}{2}}.
        \end{split}
    \end{equation*}
    For any $-3<\gamma<0$, we have $2(\gamma+2)>-3$, then by using the inequality \eqref{2-4}, it follows that
    \begin{equation*}
        \begin{split}
            \left|a_{jk}*(v_k\mu^{\frac12}f_1)\right|\lesssim\langle v\rangle^{\gamma+2}\|f_1\|_{2, \frac{\gamma}{2}}.
        \end{split}
    \end{equation*}
    Using the Cauchy-Schwarz inequality and the inequality \eqref{2-2}, we can conclude
    \begin{equation*}
        \begin{split}
            |I_1|&\lesssim\|f_1\|_{2,\frac{\gamma}{2}}\int_{\mathbb R^3}\langle v\rangle^{\gamma+3}\mu^{\frac12}(v)\left(|f_2(v)|+|\nabla f_2(v)|\right)dv\\
            &\lesssim\|f_1\|_{2,\frac{\gamma}{2}}\left(\|f_2\|_{2,1+\frac{\gamma}{2}}+\|\nabla f_2\|_{2, \frac{\gamma}{2}}\right)\lesssim\|f_1\|_{2, \frac{\gamma}{2}}\|f_2\|_{A}.
        \end{split}
    \end{equation*}
    For the term $I_2$, from \eqref{2-5}, one has
    \begin{equation*}
        \begin{split}
            \left|\partial_{jk}a_{jk}*(\mu^{\frac12}f_1)\right|\lesssim\int_{\mathbb R^3}|v-v'|^\gamma\mu^{\frac12}(v')\left|f_1(v')\right|dv'.
        \end{split}
    \end{equation*}
    Consider two set $\{|v-v'|\le 1\}$ and $\{|v-v'|\ge 1\}$, that is 
     \begin{equation*}
        \begin{split}
            \int_{\mathbb R^3}|v-v'|^\gamma\mu^{\frac12}(v')\left|f_1(v')\right|dv'=\int_{|v-v'|\le 1}+\int_{|v-v'|\ge 1}=A_{1}+A_{2}.
        \end{split}
    \end{equation*}
    For the term $A_{1}$, since $-3<\gamma<0$, we have
    \begin{equation*}
        \begin{split}
            A_{1}&=\sum_{j\le0}\int_{2^{j-1}\le|v-v'|\le2^{j}}|v-v'|^\gamma\mu^{\frac12}(v')\left|f_1(v')\right|dv'\\ 
            &\le\sum_{j\le0}\left(2^{j-1}\right)^{\gamma}\int_{|v-v'|\le2^{j}}\mu^{\frac12}(v')\left|f_1(v')\right|dv'\\ 
            &=8\sum_{j\le0}\left(2^{j-1}\right)^{\gamma+3}\frac{1}{2^{3j}}\int_{|v-v'|\le2^{j}}\mu^{\frac12}(v')\left|f_1(v')\right|dv'\\ 
            &\le8\sum_{j\le0}\left(2^{j-1}\right)^{\gamma+3}M(\mu^{\frac12}f_{1})\lesssim M(\mu^{\frac12}f_{1})
        \end{split}
    \end{equation*}
    where $M$ is the Hardy-Littlewood maximal function. For term $A_{2}$, from \eqref{2-4},
     \begin{equation*}
        \begin{split}
            A_{2}&=\int_{|v-v'|\ge 1}|v-v'|^\gamma\mu^{\frac12}(v')\left|f_1(v')\right|dv'\\ 
            &\le\int_{\mathbb R^{3}}|v-v'|^{\gamma+2}\mu^{\frac12}(v')\left|f_1(v')\right|dv'\\ 
            &\lesssim\left(\int_{\mathbb R^3}|v-v'|^{2(\gamma+2)}\mu^{\frac12}(v')dv'\right)^{\frac12}\|f_1\|_{2, \frac{\gamma}{2}}\lesssim\langle v\rangle^{\gamma+2}\|f_1\|_{2, \frac{\gamma}{2}}.
        \end{split}
    \end{equation*}
    Combining $A_{1}$ and $A_{2}$, using Cauchy-Schwarz inequality to get
    \begin{equation*}
        \begin{split}
            |I_{2}|&\lesssim\int_{\mathbb R^{3}}|M(\mu^{\frac12}f_{1})(v)\mu^{\frac12}(v)f_{2}(v)|dv+\|f_1\|_{2, \frac{\gamma}{2}}\int_{\mathbb R^{3}}\langle v\rangle^{\gamma+2}\mu^{\frac12}(v)|f_{2}(v)|dv\\ 
            &\lesssim\|M(\mu^{\frac12}f_{1})\|_{L^{2}}\|\mu^{\frac12}f_{2}\|_{L^{2}}+\|f_1\|_{2, \frac{\gamma}{2}}\|f_2\|_{2, \frac{\gamma}{2}}\\
            &\lesssim\|\mu^{\frac12}f_{1}\|_{L^{2}}\|\mu^{\frac12}f_{2}\|_{L^{2}}+\|f_1\|_{2, \frac{\gamma}{2}}\|f_2\|_{2, \frac{\gamma}{2}}\\ 
            &\lesssim\|f_1\|_{2, \frac{\gamma}{2}}\|f_2\|_{2, \frac{\gamma}{2}}\lesssim\|f_1\|_{2, \frac{\gamma}{2}}\|f_2\|_A.
        \end{split}
    \end{equation*}
    For $I_3$, from \eqref{2-5}, we have
    \begin{equation*}
        \begin{split}
            \left|\partial_ka_{jk}*(\mu^{\frac12}f_1)\right|\lesssim\int_{\mathbb R^3}|v-v'|^{\gamma+1}\mu^{\frac12}(v')|f_1(v')|dv'.
        \end{split}
    \end{equation*}
    Note that
    $$\frac32(\gamma+1)>-3,$$
    with $-3<\gamma<0$. Using H$\ddot{\rm o}$lder's inequality, \eqref{mu} and \eqref{2-4}, we have
    \begin{equation*}
        \begin{split}
            \left|\partial_ka_{jk}*(\mu^{\frac12}f_1)\right|&\lesssim\left(\int_{\mathbb R^3}|v-v'|^{\frac32(\gamma+1)}\mu^{\frac12}(v')dv'\right)^{\frac23}\|f_1\|_{3, \frac{\gamma}{2}}\lesssim\langle v\rangle^{\gamma+1}\|f_1\|_{3, \frac{\gamma}{2}}.
        \end{split}
    \end{equation*}
         Now, we want to show $\|f_1\|_{3, \frac{\gamma}{2}}$ can be bounded by $\|f_1\|_A$. By applying H$\ddot{\rm o}$lder's inequality, $\langle v\rangle^{\gamma/2}f_1(v)$ in $L^3(\mathbb R^3)$ can be bounded by
    \begin{equation*}
        \left(\|\langle\cdot\rangle^{\gamma/2}f_1\|_{L^2}\|\langle\cdot\rangle^{\gamma/2}f_1\|_{L^6}\right)^{\frac12},
    \end{equation*}
    and Sobolev embedding implies 
    $$\|\langle\cdot\rangle^{\gamma/2}f_1\|_{L^6}\lesssim\|\nabla[\langle\cdot\rangle^{\gamma/2}f_1]\|_{L^2},$$
    thus we get
    \begin{equation*}
        \|f_1\|_{3, \frac{\gamma}{2}}\lesssim\left(\|\langle\cdot\rangle^{\gamma/2}f_1\|_{L^2}\|\nabla[\langle\cdot\rangle^{\gamma/2}f_1]\|_{L^2}\right)^{\frac12}.
    \end{equation*}
    Notice that
    \begin{equation*}
        \nabla[\langle v\rangle^{\gamma/2}f_1(v)]=\frac{\gamma}{2}\langle v\rangle^{\gamma/2-2}f_1(v)v+\langle v\rangle^{\gamma/2}\nabla f_1(v),
    \end{equation*}
    from \eqref{2-2}, we have
    \begin{equation*}
        \begin{split}
            \|\nabla[\langle\cdot\rangle^{\gamma/2}f_1]\|_{L^2}\le\left\|\frac{\gamma}{2}\langle v\rangle^{\gamma/2-2}f_1v\right\|_{L^2}+\|\langle\cdot\rangle^{\gamma/2}\nabla f_1\|_{L^2}\lesssim\|f_1\|_A,
        \end{split}
    \end{equation*}
    which implies 
    \begin{equation}\label{inequality}
        \|f_1\|_{3, \frac{\gamma}{2}}\lesssim\|f_1\|_A. 
    \end{equation}
    Finally, using the Cauchy-Schwarz inequality and \eqref{inequality} to get
    \begin{equation*}
        \begin{split}
            |I_3|\lesssim\|f_1\|_A\int_{\mathbb R^3}\langle v\rangle^{\gamma+2}\mu^{\frac12}(v)|f_2(v)|dv\lesssim\|f_1\|_A\|f_2\|_{2, \frac{\gamma}{2}}.
        \end{split}
    \end{equation*}
    Combining $I_1-I_3$, we obtain
    \begin{equation*}
        |(\mathcal L_2f_1, f_2)_{L^2}|\le C_3\left(\|f_1\|_{2, \frac{\gamma}{2}}\|f_2\|_A+\|f_1\|_A\|f_2\|_{2, \frac{\gamma}{2}}\right).
    \end{equation*}
\end{proof}

\begin{remark}
    (1). For $f_1, f_2\in\mathcal S(\mathbb R^3)$ and $\gamma>-5/2$, we have
    \begin{equation*}
        |(\mathcal L_2f_1, f_2)_{L^2}|\lesssim\|f_1\|_{2, \frac{\gamma}{2}}\|f_2\|_A.
    \end{equation*}
    (2). For $-3<\gamma<0$, if $f_1=f_2$, then for any $\epsilon_2>0$, there exists a constant $C_{\epsilon_2}>0$ such that
    \begin{equation}\label{f1}
        |(\mathcal L_2f_1, f_1)_{L^2}|\le\epsilon_2\|f_1\|^2_A+C_{\epsilon_2}\|f_1\|^2_{2, \frac{\gamma}{2}}.
    \end{equation}
    (3). From \eqref{2-2}, 
    \begin{equation}\label{12}
        |(\mathcal L_2f_1, f_2)_{L^2}|\le C_4\|f_1\|_{A}\|f_2\|_A.
    \end{equation}
\end{remark}

\section{Energy Estimates}

In this section, we study the energy estimates of the solution to the Cauchy problem \eqref{1-2}.

\begin{lemma}\label{lemma3.1}
    For $-3<\gamma<0$. Lef $f$ be the solution of Cauchy problem \eqref{1-2}. Assume $f_0\in L^2(\mathbb R^3)$. Then there exists a constant $C_5>0$ such that for any $T>0$ and $t\in]0,T]$,
    \begin{equation*}
        \|f(t)\|^2_{L^2(\mathbb R^3)}+\int_0^t\|f(s)\|^2_Ads\le\left(C_5\right)^2.
    \end{equation*}
\end{lemma}

\begin{proof}
    Since $f$ is the solution of Cauchy problem \eqref{1-2},
        \begin{equation*}
            \frac12\frac{d}{dt}\|f(t)\|^2_{L^2(\mathbb R^3)}+(\mathcal L_1f, f)_{L^2(\mathbb R^3)}=(g,f)_{L^2(\mathbb R^3)}-(\mathcal L_2f, f)_{L^2(\mathbb R^3)}.
        \end{equation*}
    Since $\gamma<0$, by using Lemma \ref{lemma2.1} and \eqref{f1}, we have
        \begin{equation*}
        \begin{split}
            &\frac{d}{dt}\|f(t)\|^2_{L^2(\mathbb R^3)}+2(1-\epsilon_1)\|f(t)\|^2_A\\
            &\le2\|f(t)\|_{L^2}\|g(t)\|_{L^2}+2\epsilon_2\|f(t)\|^2_A+2(C_{\epsilon_1}+C_{\epsilon_2})\|f(t)\|^2_{L^2(\mathbb R^3)}.
        \end{split}
        \end{equation*}
    Applying Cauchy-Schwarz inequality, and choosing
    $$\epsilon_1=\epsilon_2=\frac14,$$
    we can get
        \begin{equation*}
            \begin{split}
                \frac{d}{dt}\|f(t)\|^2_{L^2(\mathbb R^3)}+\|f(t)\|^2_A\le c_1\|f(t)\|^2_{L^2}+\|g(t)\|^2_{L^2},
            \end{split}
        \end{equation*}
    Since $g$ is analytic with respect to $t$ and $v$, for any $T>0$ and $t\in]0,T]$, there exists a constant $A>0$ such that
        \begin{equation}\label{3-1}
            \|g(t)\|_{L^2}\le A.
        \end{equation}
    Therefore by applying Gronwall inequality, for any $T>0$ and $t\in]0,T]$, taking $C_5\ge e^{\frac12 c_1T}\sqrt{\|f_0\|^2_{L^2}+TA^2}$, one can obtain
        \begin{equation*}
            \begin{split}
                \|f(t)\|^2_{L^2(\mathbb R^3)}+\int_0^t\|f(s)\|^2_Ads\le e^{c_1T}\left(\|f_0\|^2_{L^2}+TA^2\right)\le (C_5)^2.
            \end{split}
        \end{equation*}
\end{proof}

\begin{lemma}\label{lemma3.2}
    For $-3<\gamma<0$. Lef $f$ be the solution of Cauchy problem \eqref{1-2}. Assume $f_0\in L^2(\mathbb R^3)$. Then there exists a constant $C_6>0$ such that for any $T>0$ and $t\in]0,T]$,
    \begin{equation}\label{k=1}
        \|t\partial_tf\|^2_{L^\infty(]0, T]; L^2(\mathbb R^3))}+\int_0^T\|t\partial_tf\|^2_Adt\le\left(C_6\right)^2.
    \end{equation}
\end{lemma}
\begin{proof}
    Since the solution of Cauchy problem \eqref{1-2} belongs to $C^\infty(]0, T[; \mathcal S(\mathbb R^3))$, we have that
    \begin{equation*}
        \partial_t(t\partial_tf)+\mathcal L_1(t\partial_tf)=\partial_tf-\mathcal L_2(t\partial_tf)+t\partial_tg,
    \end{equation*}
    and for $0<t\le T$,
    \begin{equation*}
        \begin{split}
            &\frac12\|t\partial_tf\|^2_{L^2(\mathbb R^3)}+\int_0^t(\mathcal L_1(s\partial_sf), s\partial_sf)_{L^2(\mathbb R^3)}ds\\
            &=\int_0^ts\|\partial_sf\|^2_{L^2(\mathbb R^3)}ds-\int_0^t(\mathcal L_2(s\partial_sf), s\partial_sf)_{L^2(\mathbb R^3)}+\int_0^t(s\partial_sf, s\partial_sg)_{L^2(\mathbb R^3)}ds\\
            &=S_1+S_2+S_3.
        \end{split}
    \end{equation*}
    Firstly, since $\gamma<0$, by lemma \ref{lemma2.1}, for all $0<t\le T$, we can conclude
    \begin{equation*}
        \begin{split}
            &\int_0^t(\mathcal L_1(s\partial_sf), s\partial_sf)_{L^2(\mathbb R^3)}ds\\
            &\ge(1-\epsilon_1)\int_0^t\|s\partial_sf\|^2_Ads-C_{\epsilon_1}\int_0^t\|s\partial_sf\|^2_{2, \frac{\gamma}{2}}ds\\
            &\ge(1-\epsilon_1)\int_0^t\|s\partial_sf\|^2_Ads-TC_{\epsilon_1}\int_0^ts\|\partial_sf\|^2_{L^2(\mathbb R^3)}ds.
        \end{split}
    \end{equation*}
    For the term $S_1$, since $f$ is the solution of \eqref{1-2}, using Proposition \ref{prop L} and \eqref{12}, for all $0<t\le T$, we have
    \begin{equation*}
        \begin{split}
            &\int_0^ts\|\partial_sf\|^2_{L^2(\mathbb R^3)}ds\\
            &=\int_0^t(g, s\partial_sf)_{L^2(\mathbb R^3)}ds-\int_0^t(\mathcal L_1f, s\partial_sf)_{L^2(\mathbb R^3)}ds\\
            &\qquad-\int_0^t(\mathcal L_2f, s\partial_sf)_{L^2(\mathbb R^3)}ds\\
            &\le\int_0^t\|g(s)\|_{L^2(\mathbb R^3)}\|s\partial_sf\|_{L^2(\mathbb R^3)}ds+C_2\int_0^t\|f(s)\|_A\|s\partial_sf\|_Ads\\
            &\qquad+C_4\int_0^t\|f(s)\|_A\|s\partial_sf\|_Ads.
        \end{split}
    \end{equation*}
    For all $0<t\le T$, by Cauchy-Schwarz inequality, it follows that
    \begin{equation*}
        \begin{split}
            &\int_0^t\|g(s)\|_{L^2(\mathbb R^3)}\|s\partial_sf\|_{L^2(\mathbb R^3)}ds\\
            &\le T^{\frac12}\int_0^t\|g(s)\|_{L^2(\mathbb R^3)}s^{\frac12}\|\partial_sf\|_{L^2(\mathbb R^3)}ds\\
            &\le\frac12\int_0^ts\|\partial_sf\|^2_{L^2(\mathbb R^3)}ds+\frac{T}{2}\int_0^t\|g(s)\|^2_{L^2(\mathbb R^3)}ds.
        \end{split}
    \end{equation*}
    Using Cauchy-Schwarz inequality, since $\gamma<0$, for any $0<\delta<1$, we have
    \begin{equation*}
        \begin{split}
            \int_0^ts\|\partial_sf\|^2_{L^2(\mathbb R^3)}ds&\le\delta\int_0^t\|s\partial_sf\|^2_Ads+T\int_0^t\|g(s)\|^2_{L^2(\mathbb R^3)}ds\\
            &\quad+C_\delta\int_0^t\|f(s)\|^2_{A}ds,
        \end{split}
    \end{equation*}
    with $C_\delta$ depends on $C_2, C_4$. Combining \eqref{3-1} and Lemma \ref{lemma3.1}, for $0<\delta<1$,
    \begin{equation}\label{3.2}
        \begin{split}
            S_1&=\int_0^ts\|\partial_sf\|^2_{L^2(\mathbb R^3)}ds\le\delta\int_0^t\|s\partial_sf\|^2_Ads+T^2A^2+C_\delta C^2_5.
        \end{split}
    \end{equation}
    For the term $S_2$, let $f_1=s\partial_sf$ in \eqref{f1}, then for all $0<t\le T$, 
    \begin{equation*}
        \begin{split}
            |S_2|\le\epsilon_2\int_0^t\|s\partial_sf\|^2_Ads+C_{\epsilon_2}T\int_0^ts\|\partial_sf\|^2_{L^2(\mathbb R^3)}ds,
        \end{split}
    \end{equation*}
    by using \eqref{3.2} with $c_3T\delta\le\epsilon_2$,
    \begin{equation*}
        |S_2|\le2\epsilon_2\int_0^t\|s\partial_sf\|^2_Ads+\tilde C_{\epsilon_2},
    \end{equation*}
    with $\tilde C_{\epsilon_2}$ depends on $C_2, C_4, C_5, A$ and $T$.

    Finally, for the term $S_3$, by Cauchy-Schwarz inequality, it follows that
    \begin{equation*}
        \begin{split}
            |S_3|\le\int_0^t\|s\partial_sg\|^2_{L^2(\mathbb R^3)}+T\int_0^ts\|\partial_sf\|^2_{L^2(\mathbb R^3)}ds.
        \end{split}
    \end{equation*}
    Since $g$ is analytic with respect to $t$ and $v$, for all $0<t\le T$, we have
    \begin{equation*}
        \|t\partial_tg\|_{L^2(\mathbb R^3)}\le A^2,
    \end{equation*}
    applying \eqref{3.2} with $T\delta\le\epsilon_2$ to get $S_3$ can be bounded by
    \begin{equation*}
        \epsilon_2\int_{0}^t\|s\partial_sf\|^2_Ads+TA^4+T\left(T^2A^2+C_{\epsilon_2}C^2_5\right).
    \end{equation*}
    Therefore, combining the results above and using \eqref{3.2} with $TC_{\epsilon_1}\delta<\epsilon_1$, let 
    $$\epsilon_1=\epsilon_2=\frac{1}{16},\quad0<\delta\le\frac14,$$
    and taking $C_6\ge\sqrt{\tilde C_5}$, we get
    \begin{equation*}
        \begin{split}
            \|t\partial_tf\|^2_{L^\infty(]0, T]; L^2(\mathbb R^3))}+\int_0^T\|t\partial_tf\|^2_Adt\le\tilde C_5\le\left(C_6\right)^2.
        \end{split}
    \end{equation*}
    with $C_6$ depend on $C_2, C_4, C_5, A$ and $T$.
\end{proof}

\section{Analytic Smoothing Effect for Time Variable}

In this section, we will show the analytic regularity of the
time variable for $t>0$. We construct the following estimate, which implies Theorem \ref{thm} immediately .

\begin{prop}\label{prop 4.1}
    For $-3<\gamma<0$. Let $f$ be the solution of Cauchy problem \eqref{1-2}, and $f_0\in L^2(\mathbb R^3)$. Then there exists a constant $B>0$ such that for any $T>0$, $t\in]0,T]$ and $k\in\mathbb N$,
    \begin{equation}\label{4-1}
        \|t^k\partial_t^kf\|^2_{L^\infty(]0, T]; L^2(\mathbb R^3))}+\int_0^T\|t^k\partial_t^kf\|^2_Adt\le\left(B^{k+1}k!\right)^2.
    \end{equation}
\end{prop}
\begin{proof}
    We prove this proposition by induction on the index $k$. For $k=1$, it is enough to take in \eqref{k=1}. Assume \eqref{4-1} holds true, for any $1\le m\le k-1$ with $k\ge2$,
    \begin{equation}\label{4-2}
        \|t^m\partial_t^mf\|^2_{L^\infty(]0, T]; L^2(\mathbb R^3))}+\int_0^T\|t^m\partial_t^mf\|^2_Adt\le\left(B^{m+1}m!\right)^2.
    \end{equation}
    We shall prove \eqref{4-1} holds true for $m=k$.

    Since $\mu$ is the function with respect to $v$, which implies
    \begin{equation*}
        t^k\partial_t^k\mathcal L_1f=\mathcal L_1(t^k\partial_t^kf), \quad t^k\partial_t^k\mathcal L_2f=\mathcal L_2(t^k\partial_t^kf).
    \end{equation*}
    Then by \eqref{1-2}, we have
    \begin{equation*}
        \begin{split}
            \partial_t(t^k\partial_t^kf)+\mathcal L_1(t^k\partial_t^kf)=kt^{k-1}\partial_t^kf-\mathcal L_2(t^k\partial_t^kf)+t^k\partial_t^kg.
        \end{split}
    \end{equation*}
    Taking the $L^2(\mathbb R^3)$ inner product of both sides with respect to $t^k\partial_t^kf$, we get
    \begin{equation*}
        \begin{split}
            &\frac12\frac{d}{dt}\|t^k\partial_t^kf\|^2_{L^2(\mathbb R^3)}+(\mathcal L_1(t^k\partial_t^kf) ,t^k\partial_t^kf)_{L^2(\mathbb R^3)}\\
            &=k(t^{k-1}\partial_t^kf, t^{k}\partial_t^kf)_{L^2(\mathbb R^3)}-(\mathcal L_2(t^k\partial_t^kf) ,t^k\partial_t^kf)_{L^2(\mathbb R^3)}\\
            &\quad+(t^{k}\partial_t^kg, t^{k}\partial_t^kf)_{L^2(\mathbb R^3)}.
        \end{split}
    \end{equation*}
    For all $0<t\le T$, integrating from 0 to $t$, since $\gamma<0$, by using lemma \ref{lemma2.1}, it follows that
    \begin{equation*}
        \begin{split}
            &\int_0^t(\mathcal L_1(s^k\partial_s^kf) ,s^k\partial_s^kf)_{L^2(\mathbb R^3)}ds\\
            &\ge(1-\epsilon_1)\int_0^t\|s^k\partial_s^kf\|^2_Ads-C_{\epsilon_1}\int_0^t\|s^k\partial_s^kf\|^2_{2, \frac{\gamma}{2}}ds\\
            &\ge(1-\epsilon_1)\int_0^t\|s^k\partial_s^kf\|^2_Ads-TC_{\epsilon_1}\int_0^ts^{2k-1}\|\partial_s^kf\|^2_{L^2(\mathbb R^3)}ds,
        \end{split}
    \end{equation*}
    and let $f_1=s^k\partial_s^kf$ in \eqref{f1} to get
    \begin{equation*}
        \begin{split}
            \int_0^t|(\mathcal L_2(s^k\partial_s^kf) ,&s^k\partial_s^kf)_{L^2(\mathbb R^3)}|ds\le\epsilon_2\int_0^t\|s^k\partial_s^kf\|^2_Ads\\
            &\quad+TC_{\epsilon_2}\int_0^ts^{2k-1}\|\partial_s^kf\|^2_{L^2(\mathbb R^3)}ds,
        \end{split}
    \end{equation*}
    then using Cauchy-Schwarz inequality to get
    \begin{equation*}
        \begin{split}
            &\int_0^t|(s^{k}\partial_s^kg, s^{k}\partial_s^kf)_{L^2(\mathbb R^3)}|ds\\
            &\le\frac12\int_0^t\|s^{k}\partial_s^kg\|^2_{L^2(\mathbb R^3)}ds+\frac12\int_0^t\|s^{k}\partial_s^kf\|^2_{L^2(\mathbb R^3)}ds\\
            &\le\frac12\int_0^t\|s^{k}\partial_s^kg\|^2_{L^2(\mathbb R^3)}ds+\frac{T}{2}\int_0^ts^{2k-1}\|\partial_s^kf\|^2_{L^2(\mathbb R^3)}ds.
        \end{split}
    \end{equation*}
    Combining the results above, and taking
    $$\epsilon_1=\epsilon_2=\frac18,$$
    we have for all $0<t\le T$,
    \begin{equation}\label{4-3}
        \begin{split}
            &\frac{d}{dt}\|t^k\partial_t^kf\|^2_{L^2(\mathbb R^3)}+\frac32\int_0^t\|s^k\partial_s^kf\|^2_Ads\\
            &\le\int_0^t\|s^{k}\partial_s^kg\|^2_{L^2(\mathbb R^3)}ds+C_7\int_0^ts^{2k-1}\|\partial_s^kf\|^2_{L^2(\mathbb R^3)}ds,
        \end{split}
    \end{equation}
    with $C_7$ depends on $T$.

    Since $f$ is the solution of \eqref{1-2} and $k\ge2$, we have
    $$
    \partial_t^kf=\partial_t^{k-1}g-\mathcal L_1(\partial_t^{k-1}f)-\mathcal L_2(\partial_t^{k-1}f),$$
    which implies
    \begin{equation*}
        \begin{split}
            \int_0^t&s^{2k-1}\|\partial_s^kf\|^2_{L^2(\mathbb R^3)}ds=\int_0^t(s^{k-1}\partial_s^{k-1}g, s^k\partial_s^kf)_{L^2(\mathbb R^3)}ds\\
            &-\int_0^t(\mathcal L_1(s^{k-1}\partial_s^{k-1}f), s^k\partial_s^kf)_{L^2(\mathbb R^3)}ds-\int_0^t(\mathcal L_2(s^{k-1}\partial_s^{k-1}f), s^k\partial_s^kf)_{L^2(\mathbb R^3)}ds.
        \end{split}
    \end{equation*}
    For all $0<t\le T$, using Cauchy-Schwarz inequality, we have
    \begin{equation*}
        \begin{split}
            &\int_0^t(s^{k-1}\partial_s^{k-1}g, s^k\partial_s^kf)_{L^2(\mathbb R^3)}ds\\
            &\le T^{\frac12}\int_0^t\|s^{k-1}\partial_s^{k-1}g\|_{L^2(\mathbb R^3)}s^{\frac{2k-1}{2}}\|\partial_s^kf\|_{L^2(\mathbb R^3)}\\
            &\le\frac{T}{2}\int_0^t\|s^{k-1}\partial_s^{k-1}g\|^2_{L^2(\mathbb R^3)}+\frac12\int_0^ts^{2k-1}\|\partial_s^kf\|^2_{L^2(\mathbb R^3)}ds.
        \end{split}
    \end{equation*}
    By using Proposition \ref{prop L}, \eqref{12} and Cauchy-Schwarz inequality, for any $0<\delta<1$, there exists a constant $C_\delta>0$ such that for all $0<t\le T$,
    \begin{equation}\label{4-4}
        \begin{split}
            \int_0^ts^{2k-1}\|\partial_s^kf\|^2_{L^2(\mathbb R^3)}ds&\le\delta\int_0^t\|s^{k}\partial_s^{k}f\|^2_{A}ds +T\int_0^t\|s^{k-1}\partial_s^{k-1}g\|^2_{L^2(\mathbb R^3)}\\
            &\quad+C_\delta\int_0^t\|s^{k-1}\partial_s^{k-1}f\|^2_Ads,
        \end{split}
    \end{equation}
    with $C_\delta$ depends on $C_2, C_4$. Let $C_7\delta\le\frac12$, substituting \eqref{4-4} into \eqref{4-3}, we get
    \begin{equation*}
        \begin{split}
            &\frac{d}{dt}\|t^k\partial_t^kf\|^2_{L^2(\mathbb R^3)}+\int_0^t\|s^k\partial_s^kf\|^2_Ads\\
            &\le\tilde C_7\left(\int_0^t\|s^{k-1}\partial_s^{k-1}g\|^2_{L^2(\mathbb R^3)}ds+\int_0^t\|s^{k-1}\partial_s^{k-1}f\|^2_Ads\right)\\
            &\qquad+\int_0^t\|s^k\partial_s^kg\|^2_{L^2(\mathbb R^3)}ds,
        \end{split}
    \end{equation*}
    with $\tilde C_7$ depends on $C_2, C_4, C_7$ and $T$.

    Finally, since $g$ is analytic with respect to $t$ and $v$, for any $k\in\mathbb N$, there exists a constant $A>0$ such that for any $0<t\le T$,
    \begin{equation*}
        \|t^k\partial_t^kg\|_{L^2(\mathbb R^3)}\le A^{k+1}k!,
    \end{equation*}
    taking $B\ge\max\{A,\sqrt{2\tilde C_7}\}$, using the induction hypothesis \eqref{4-2}, we obtain
    \begin{equation*}
        \begin{split}
            &\frac{d}{dt}\|t^k\partial_t^kf\|^2_{L^2(\mathbb R^3)}+\int_0^t\|s^k\partial_s^kf\|^2_Ads\\
            &\le\tilde C_7\left((A^{k}(k-1)!)^2+(B^{k}(k-1)!)^2\right)+(A^{k+1}k!)^2\\
            &\le(B^{k+1}k!)^2,
        \end{split}
    \end{equation*}
    with $B$ depends on $C_1, C_2, C_4, A$ and $T$. We finish the proof of proposition \ref{prop 4.1}.
\end{proof}

\bigskip
\noindent {\bf Acknowledgements.}
This work was supported by the NSFC (No.12031006) and the Fundamental
Research Funds for the Central Universities of China.


\begin{thebibliography}{99}

\bibitem{C-1} H. Chen, W. X. Li and C. J. Xu, Analytic smoothness effect of solutions for spatially homogeneous Landau equation. {\em J. Differential Equations}, 248 (2009), 77-94.

\bibitem{C-2} H. Chen, W. X. Li and C. j. Xu, Propagation of Gevrey regularity for solutions of Landau equations, {\em Kinet. Relat. Models}, 1(2008), 355-368.

\bibitem{C-3} Y. M. Chen, L. Desvillettes and L. B. He, Smoothing Effects for Classical Solutions of the Full Landau Equation. {\em Arch. Ration. Mech. Anal.}, 193 (2009), 391-434.

\bibitem{C-4} K. Carrapatoso, I. Tristani adn C. K. Wu, Cauchy problem and exponential stability for the inhomogeneous Landau equation. {\em Arch. Ration. Mech. Anal.}, 221 (2016), 363-418.

\bibitem{C-5} K. Carrapatoso, S. Mischler, Landau equation for very soft and Coulomb potentials near Maxwellians. {\em Ann. PDE}, 3 (2017), Paper No. 1, 65 pp.

\bibitem{D-1} L. Desvillettes, On asymptotics of the Boltzmann equation when the collision become grazing. {\em Transport Theory Statist. Phys.}, 21 (1992), 259-276.

\bibitem{D-2} Desvillettes L. and Villani C., On the spatially homogeneous landau equation for hard potentials part i: existence, uniqueness and smoothness. {\em Communications in Partial Differential Equations}, 25 (2000), 179-259.

\bibitem{G-1} Guo Y., The Landau Equation in a Periodic Box. {\em Comm. Math. Phys.}, 231 (2002), 391-434.

\bibitem{H-1} Henderson C. and Snelson S., $C^\infty$ Smoothing for Weak Solutions of the Inhomogeneous Landau Equation. {\em Arch. Ration. Mech. Anal.}, 236 (2020), 113-143.

\bibitem{L-1} Lerner N.,  Morimoto Y.,  Pravda-Starov K. and Xu C.-J., Phase space analysis and functional calculus for the linearized Landau and Boltzmann operators. {\em Kinet. Relat. Models}, 6 (2013), 625-648.

\bibitem{L-2} H. G. Li and C. J. Xu, Cauchy problem for the spatially homogeneous Landau equation with Shubin class initial datum and Gelfand- Shilov smoothing effect. {\em SIAM J. Math. Anal.}, 51 (2019), 532-564.

\bibitem{L-3} S. Liu and X. Ma, Regularizing effects for the classical solutions to the Landau equation in the whole space. {\em J. Math. Anal. Appl.}, 417 (2014), 123-143.

\bibitem{L-4} H. G. Li and C. J. Xu, The analytic smoothing effect of solutions for the nonlinear spatially homogeneous Landau equation with hard potentials.{\em  Sci China Math.}, 2021, 64. https://doi.org/10.1007/s11425-021-1888-6

\bibitem{L-5} H. G. Li and C. J. Xu, Analytic smoothing effect of linear Landau equation with soft potential. 2022, arXiv.org.

\bibitem{M-1} Morimoto Y., Pravda-Starov K. and Xu C.-J., A remark on the ultra-analysis smoothing properties of the spatially homogeneous Landau equation. {\em  Kinet. Relat. Models}, 6 (2013), 715-727.

\bibitem{M-2} Morimoto Y. and C. J. Xu, Analytic smoothing effect of the nonlinear Landau equation of Maxwellian molecules. {\em Kinet. Relat. Models}, 13 (2020), 951-983.

\bibitem{S-1} R. M. Strain and Y. Guo, Exponential Decay for Soft Potentials near Maxwellian. {\em Arch. Ration. Mech. Anal.}, 187 (2008), 287-339.

\bibitem{V-1} C. Villani, On a New Class of Weak Solutions to the Spatially Homogeneous Boltzmann and Landau Equations.{\em Arch. Rational Mech. Anal.}, 143 (1998), 273-307.

\bibitem{V-2} C. Villani, On the spatially homogeneous Landau equation for Maxwellian molecules. {\em Math. Models Methods Appl. Sci.}, 8 (1998), 957-983.
\end{thebibliography}
\end{document}